\documentclass{amsart}

\usepackage{amsmath}
\usepackage{amsthm}
\usepackage{amssymb}
\usepackage{amsfonts}
\usepackage{amsrefs}
\usepackage{color}
\usepackage{tikz}

\tikzstyle{shaded}=[fill=red!10!blue!20!gray!30!white]
\tikzstyle{shaded line}=[double=red!10!blue!20!gray!30!white, double distance=1.5mm, draw=black]
\tikzstyle{unshaded}=[fill=white]
\tikzstyle{unshaded line}=[double=white, double distance=1.5mm, draw=black]
\tikzstyle{Tbox}=[circle, draw, thick, fill=white, opaque,]
\tikzstyle{empty box}=[circle, draw, thick, fill=white, opaque, inner sep=2mm]
\tikzstyle{background rectangle}= [fill=red!10!blue!20!gray!40!white,rounded corners=2mm] 
\tikzstyle{on}=[very thick, red!50!blue!50!black]
\tikzstyle{off}=[gray]

\tikzstyle{traces}=[scale=.2, inner sep=1mm]
\tikzstyle{quadratic}=[scale=.4, inner sep=1mm, baseline]
\tikzstyle{annular}=[scale=.7, inner sep=1mm, baseline]
\tikzstyle{make triple edge size}= [scale=.4, inner sep=1mm,baseline] 
\tikzstyle{icosahedron network}=[scale=.3, inner sep=1mm, baseline]
\tikzstyle{ATLsix}=[scale=.25, baseline]
\tikzstyle{TL12}=[scale=.15,baseline]
\tikzstyle{PAdefn}=[scale=.7,baseline]
\tikzstyle{TLEG}=[scale=.5,baseline]

\makeatletter

\newtheorem{lemma}{Lemma}[section]
\newtheorem{definition}[lemma]{Definition}
\newtheorem{theorem}[lemma]{Theorem}
\newtheorem{proposition}[lemma]{Proposition}
\newtheorem{remark}[lemma]{Remark}
\newtheorem{corollary}[lemma]{Corollary}

\newtheorem{notations}[lemma]{Notations}

\newenvironment{claim}[1]{\par\noindent\underline{Claim:}\space#1}{}
\newenvironment{claimproof}[1]{\par\noindent\underline{Proof:}\space#1}{\hfill $\blacksquare$}

\DeclareMathOperator{\tr}{tr}
\DeclareMathOperator{\id}{id}
\DeclareMathOperator{\im}{im}

\makeatother

\title[Ore's theorem on subfactor planar algebras]{Ore's theorem on subfactor planar algebras}
\author[Sebastien Palcoux]{Sebastien Palcoux}
\address{Institute of Mathematical Sciences, Chennai, India}
\email{sebastienpalcoux@gmail.com}

\begin{document}

\begin{abstract}
This article proves that an irreducible subfactor planar algebra with a distributive biprojection lattice admits a minimal $2$-box projection generating the identity biprojection. It is a generalization (conjectured in 2013) of a theorem of \O ystein Ore on distributive intervals of finite groups (1938), and a corollary of a natural subfactor extension of a conjecture of Kenneth S. Brown in algebraic combinatorics (2000). We deduce a link between combinatorics and representations in finite group theory. 
\end{abstract}

\maketitle

\tableofcontents

\section{Introduction} 

Any finite group $G$ acts outerly on the hyperfinite ${\rm II}_1$ factor $R$, and the group subfactor $(R\subseteq R\rtimes G)$, of index $|G|$, remembers the group \cite{jo}. Jones proved in \cite{jo2} that the set of possible values for the index $ |M:N| $ of a subfactor $ (N \subseteq M) $ is 
 $$ \{ 4 cos^2(\frac{\pi}{n}) \ \vert \ n \geq 3 \} \sqcup [4,\infty]. $$
By Galois correspondence \cite{nk}, the lattice of intermediate subfactors of $(R\subseteq R\rtimes G)$ is isomorphic to the subgroup lattice of $G$.
Moreover, Watatani \cite{wa} extended the finiteness of the subgroup lattice to any irreducible finite index subfactor. Then, the subfactor theory can be seen as an augmentation of the finite group theory, where the indices are not necessarily integers. The notion of subfactor planar algebra \cite{jo4} is a diagrammatic axiomatization of the standard invariant of a finite index ${\rm II}_1$ subfactor \cite{js}. Bisch \cite{bi} proved that the intermediate subfactors are given by the biprojections (see Definition \ref{debi}) in the $2$-box space of the corresponding planar algebra. The recent results of Liu \cite{li} on the biprojections are also crucial for this article (see Section \ref{bipro}).

\O ystein Ore proved in 1938 that a finite group is cyclic if and only if its subgroup lattice is distributive, and he extended one way as follows:

\begin{theorem}[\cite{or}, Theorem 7] \label{oreintro}
Let $[H,G]$ be a distributive interval of finite groups. Then there is $g \in G$ such that $\langle Hg \rangle = G$.
\end{theorem}

\noindent This article generalizes Ore's Theorem \ref{oreintro} to planar algebras as follows: 

\begin{theorem} \label{mainintro}
Let $\mathcal{P}$ be an irreducible subfactor planar algebra with a distributive biprojection lattice. Then there is a minimal $2$-box projection generating the identity biprojection (it is called w-cyclic).
\end{theorem}

\noindent In general, we deduce a non-trivial upper bound for the minimal number of minimal projections generating the identity biprojection. Note that Theorem \ref{mainintro} was conjectured for the first time in a conference of the author in 2013\footnote{Annual meeting of noncommutative geometry, Caen University, December 5th, 2013.} and lastly in \cite[Conjecture 5.11]{p1}. The following application is a dual version of Theorem \ref{oreintro}. See Definition \ref{fixstab} for the notations.
\begin{theorem} \label{dualoreintro}  Let $[H,G]$ be a distributive interval of finite groups. Then $\exists V$ irreducible complex representation of $G$ such that $G_{(V^H)}=H$.
\end{theorem}
\noindent Next, we deduce a non-trivial upper bound for the minimal number of irreducible components for a faithful complex representation of $G$, involving the subgroup lattice only. This is a new link between combinatorics and representations in finite group theory.  
 
Finally, the appendix proves that in the irreducible depth $2$ case, the coproduct of two minimal central projections is given by the fusion rule of the corresponding irreducible complex representations.

This article generalizes results from finite group theory to subfactor theory (as for \cites{bdlr,p1,p3,jlw,wa,xu1,xu2}), applying back to new results in finite group theory (which is quite rare). An expert in group theory suggested the author to write a group theoretic translation of the proof of these applications \cite{p4}.
Otherwise, the author investigated (with Mamta Balodi) an other approach for a direct proof of these applications, related to a problem in algebraic and geometric combinatorics, ``essentially" due to K.S. Brown, asking whether the M\"obius invariant of the bounded coset poset $P$ of a finite group (which is equal to the reduced Euler characteristic of the order complex of the proper part of $P$) is nonzero (\cite[page 760]{sw} and \cite[Question 4]{br}). These investigations gave rise to \cite{bp}. In fact, these applications are a consequence of a \emph{relative version} of Brown's problem. Shareshian and Woodroofe proved in \cite{sw} an other consequence of Brown's problem. In \cite[Section 6]{p3}, the author extended Brown's problem to any irreducible subfactor planar algebra and explained in details how it implies Theorem \ref{mainintro}.
  
 For the convenience of the reader and because this article proves the optimal version of Ore's theorem on irreducible subfactor planar algebras, we will reproduce some preliminaries of \cites{p1,p3}, for being quite self-contained.


\normalsize
\section{Basics on lattice theory}
A \emph{lattice} $(L, \wedge , \vee)$ is a poset $L$  in which every two elements $a,b$ have a unique supremum (or \emph{join}) $a \vee b$ and a unique infimum (or \emph{meet}) $a \wedge b$. Let $G$ be a finite group. The set of subgroups $ K \subseteq G$ forms a lattice,  denoted by $\mathcal{L}(G)$, ordered by $\subseteq$, with $K_1 \vee K_2 = \langle K_1,K_2 \rangle$ and $K_1 \wedge K_2 =  K_1 \cap K_2 $. A \emph{sublattice} of $(L, \wedge , \vee)$ is a subset $L' \subseteq L$ such that $(L', \wedge , \vee)$ is also a lattice. If $a,b \in L$ with $a \le b$, then the \emph{interval} $[a,b]$ is the sublattice $\{c \in L \ \vert \ a \le c \le b \}$. Any finite lattice is \emph{bounded}, i.e. admits a minimum and a maximum, denoted by $\hat{0}$ and $\hat{1}$. The \emph{atoms} are the minima of $L \setminus \{\hat{0}\}$. The \emph{coatoms} are the maxima of $L \setminus \{\hat{1}\}$. Consider a finite lattice, $b$ the join of its atoms and $t$ the meet of its coatoms, then let call $[\hat{0},b]$ and $[t,\hat{1}]$ its \emph{bottom} and \emph{top intervals}. A lattice is \emph{distributive} if the join and meet operations distribute over each other.
\noindent  A distributive bounded lattice is called \emph{Boolean} if any element $a$ admits a unique \emph{complement} $a^{\complement}$ (i.e. $a \wedge a^{\complement} = \hat{0}$ and $a \vee a^{\complement} = \hat{1}$).

\begin{lemma} \label{compl} Let $a$ and $b$ be two elements of a Boolean lattice. If $a \vee b = \hat{1}$ then $b \ge a^{\complement}$. In particular, if $a$ is an atom then $b \in \{a^{\complement},\hat{1}\}$.
\end{lemma}
\begin{proof} It is immediate after the following computation:
$$a^{\complement} = a^{\complement} \wedge \hat{1} = a^{\complement} \wedge (a \vee b) = (a^{\complement} \wedge a) \vee (a^{\complement} \wedge b) = a^{\complement} \wedge b. \qedhere $$
\end{proof}

The subset lattice of $\{1,2, \dots, n \}$, with union and intersection, is called the Boolean lattice $\mathcal{B}_n$ of rank $n$. Any finite Boolean lattice is isomorphic to some $\mathcal{B}_n$.  A lattice is called \emph{top} (resp. \emph{bottom}) \emph{Boolean} if its top (resp. bottom) interval is Boolean. We refer to \cite{sta} for more details. 

\begin{proposition} \label{topBn}
A finite distributive lattice is top and bottom Boolean.
\end{proposition} 
\begin{proof}
See \cite[items a-i p254-255]{sta} which uses Birkhoff's representation theorem (a finite lattice is distributive if and only if it embeds into some $\mathcal{B}_n$).
\end{proof} 

\section{Ore's theorem on intervals of finite groups}

We will give our short alternative proof of Theorem \ref{oreintro} by extending it to any top Boolean interval (see Proposition \ref{topBn}). The proof of the second claim below is different from that of \cite[Theorem 2.5]{p1}, for being a correct translation of the proof of Theorem \ref{mainsubfactor}. This single variation reveals how an extension of \cite{p1} was possible. 

 \begin{definition}  \label{Hcy}
An interval of finite groups $[H,G]$ is called \emph{$H$-cyclic} if there is $g \in G$ such that $\langle Hg \rangle = G$. \emph{Note that $\langle Hg \rangle = \langle H,g \rangle$}.
\end{definition}

\begin{theorem} \label{ore}
A top Boolean interval $ [H,G] $ is $H$-cyclic.
\end{theorem} 
\begin{proof}
The proof follows from the claims below.
\\
\begin{claim} \label{max} Let $ M $ be a maximal subgroup of $ G $. Then $ [M,G] $ is $ M $-cyclic.
\end{claim}
\begin{claimproof}
For $ g \in G $ with $ g \not \in M $, we have $ \langle  M,g \rangle = G $ by maximality.
\end{claimproof}
\\
\begin{claim} \label{preore2}
A Boolean interval $ [H,G] $ is $ H $-cyclic.
\end{claim} 
\begin{claimproof}
Let $ K $ be an atom in $ [H,G] $. By induction on the rank of the Boolean lattice (initiated by the previous claim), we can assume $ [K,G] $ to be $K$-cyclic, i.e. there is $g \in G$ such that $ \langle  K,g \rangle = G $. Now, for all $g' \in Kg$ we have $$ \langle  K,g \rangle = \langle  Kg \rangle = \langle  Kg' \rangle  = \langle  K,Hg' \rangle. $$ But $Kg$ decomposes into a finite partition of $H$-cosets $Hg_i$ with $i=1, \dots, |K:H|$. It follows that for all $i$, we have $ K \vee \langle  Hg_i \rangle = G, $ and so $\langle  Hg_i \rangle \in \{K^{\complement},G\}$ by Lemma \ref{compl}. If for all $i$ we have $\langle  Hg_i \rangle = K^{\complement}$, then $$ G = \langle  Kg \rangle = \langle  \bigsqcup_i Hg_i \rangle = \bigvee_i \langle  Hg_i \rangle = \bigvee_i K^{\complement} = K^{\complement},$$ which is a contradiction. So there is $i$ such that $\langle  Hg_i \rangle = G$. The result follows.
\end{claimproof} 
\\
\begin{claim} \label{topred} 
 $ [H , G] $ is $ H $-cyclic if its top interval $ [K,G] $ is $K$-cyclic.
\end{claim}
\begin{claimproof}
Consider $ g \in G $ with $ \langle  K,g \rangle = G $. For any coatom $ M \in [H,G] $, we have $ K \subseteq M $ by definition, and so $ g \not \in M $, then a fortiori $ \langle  H,g \rangle \not \subseteq M $. It follows that $ \langle  H,g \rangle=G $.
\end{claimproof} 
\end{proof}

\noindent The converse is false because $\langle S_2, (1234) \rangle = S_4$ whereas $[S_2,S_4]$ is not top Boolean.

\section{Biprojections and basic results} \label{bipro}
For the notions of subfactor, subfactor planar algebra and basic properties, we refer to \cites{js,jo4,sk2}. See also \cite[Section 3]{p0} for a short introduction. Let $(N \subseteq M)$ be a finite index irreducible subfactor. The $n$-box spaces $\mathcal{P}_{n,+}$ and $\mathcal{P}_{n,-}$ of the planar algebra $\mathcal{P}=\mathcal{P}(N \subseteq M)$, are $N' \cap M_{n-1}$ and $M' \cap M_{n}$. A \emph{projection} is an operator $p$ such that $p=p^2=p^{\star}$. Let $N \subseteq K \subseteq M$ be an intermediate subfactor. Then, the \emph{Jones projection} $e^M_K: L^2(M) \to L^2(K)$ is an element of $\mathcal{P}_{2,+}$. Consider $e_1 := e^M_N$ and $\id:=e^M_M $ the identity. Note that $\tr(\id) = 1$ and $\tr(e_1) = |M:N|^{-1} = \delta^{-2}$. Let $\langle a \vert b \rangle:= \tr(b^{\star}a)$ be the inner product of $a$ and $b \in \mathcal{P}_{2,\pm}$. Let $\mathcal{F}: \mathcal{P}_{2,\pm} \to  \mathcal{P}_{2,\mp}$ be the \emph{Fourier transform}, and let $a*b$ be the \emph{coproduct} of $a$ and $b$. Then $a * b = \mathcal{F}(\mathcal{F}^{-1}(a) \mathcal{F}^{-1}(b)).$ Note that $a*e_1 = e_1 * a = \delta^{-1}a$ and $a * \id = \id * a = \delta \tr(a) \id$. Let $\overline{a}:=\mathcal{F}(\mathcal{F}(a))$ be the \emph{contragredient} of $a$. Let $R(a)$ be the range projection of $a$.  We define the relations $a \preceq b$ and $a \sim b$ by $R(a) \le R(b)$ and $R(a) = R(b)$, respectively.
\begin{lemma}    \label{pre2}
Let $p,q \in \mathcal{P}_{2,+}$ be projections. Then  $$e_1 \preceq p *\overline{q} \Leftrightarrow pq \neq 0.$$ 
\end{lemma}
\begin{proof} The result follows by irreducibility (i.e. $\mathcal{P}_{1,+}=\mathbb{C}$).
\end{proof}
\noindent Note that if $p \in \mathcal{P}_{2,+}$ is a projection then $\overline{p}$ is also a projection.

\begin{definition}[\cites{bi,la,li}] \label{debi}
A \emph{biprojection} is a projection $b \in \mathcal{P}_{2,+} \setminus \{ 0\}$ with $\mathcal{F}(b)$ a multiple of a projection. 
\end{definition}
\noindent Note that $e_1=e^M_N$ and $\id=e^M_M$ are biprojections.
\begin{theorem}[\cite{bi} p212] \label{bisch}
A projection $b \in \mathcal{P}_{2,+}$ is a biprojection if and only if it is the Jones projection $e^M_K$ of an intermediate subfactor $N \subseteq K \subseteq M$. 
\end{theorem}
\noindent Then, the set of biprojections is a finite lattice \cite{wa}, of the form $[e_1,\id]$.
 \begin{theorem} \label{biproj} An operator $b \in \mathcal{P}_{2,+}$ is a biprojection if and only if
$$e_1 \le b=b^2=b^{\star}=\overline{b} \sim b * b.$$
Moreover, $b*b = \delta \tr(b)b$.
\end{theorem}
\begin{proof} See \cite[items 0-3 p191]{la} and \cite[Theorem 4.12]{li}.
\end{proof}

\begin{lemma} \label{prodcoprod} Consider $a_1,a_2,b \in \mathcal{P}_{2,+}$ with $b$ a biprojection, then 
$$(b \cdot a_1 \cdot b) * (b \cdot a_2 \cdot b) = b \cdot (a_1 * (b \cdot a_2 \cdot b)) \cdot b = b \cdot ((b \cdot a_1  \cdot b) * a_2) \cdot b$$
$$(b*a_1*b) \cdot (b*a_2*b) = b*(a_1 \cdot (b*a_2*b))*b = b*((b*a_1*b) \cdot a_2)*b$$  
\end{lemma}
\begin{proof} 
By exchange relations \cite{la} on $b$ and $\mathcal{F}(b)$.
    \end{proof}
Now, we define the biprojection generated by a positive operator.
\begin{definition} \label{gener}
Consider $a \in \mathcal{P}_{2,+}$ positive, and let $p_n$ be the range projection of $\sum_{k=1}^n a^{*k}$. By finiteness there exists $N$ such that for all $m \ge N$, $p_m = p_N$, which is a biprojection \cite[Lemma 4.14]{li}, denoted $\langle a \rangle$, called the \emph{biprojection generated} by $a$. It is the smallest biprojection $b \succeq a$. For $S$ a finite set of elements in $\mathcal{P}_{2,+}$, let  $\langle S \rangle $ be $ \langle \sum_{s \in S}R(s) \rangle$. 
 
\end{definition}

\begin{lemma} \label{th} Let $a,b,c,d $ be positive operators of $ \mathcal{P}_{2,+}$. Then
\begin{itemize}
\item[(1)] $a*b$ is also positive, 
\item[(2)] $[a \preceq b$ and $c \preceq d]$ $\Rightarrow$ $a*c \preceq b*d$,
\item[(3)] $ a \preceq b \Rightarrow  \langle a \rangle \le \langle b \rangle$,
\item[(4)] $ a \sim b \Rightarrow  \langle a \rangle = \langle b \rangle$.
\end{itemize} 
\end{lemma}
\begin{proof}
It's precisely \cite[Theorem 4.1 and Lemma 4.8]{li} for (1) and (2). Next, if $a \preceq b$ then by (2), for any integer $k$, $a^{*k} \preceq b^{*k}$, so for any $n$, $$\sum_{k=1}^n a^{*k} \preceq \sum_{k=1}^n b^{*k},$$ then  $\langle a \rangle \le \langle b \rangle$ by Definition \ref{gener}. Finally, (4) is immediate from (3).
\end{proof}

Let $N \subseteq K \subseteq M$ be an intermediate subfactor. The planar algebras $\mathcal{P}(N \subseteq K)$ and $\mathcal{P}(K \subseteq M)$ can be derived from $\mathcal{P}(N \subseteq M)$, see \cites{kcb,la0}. 
\begin{theorem} \label{2iso}
Consider the intermediate subfactors $N \subseteq P \subseteq K \subseteq Q \subseteq M$. Then there are two isomorphisms of von Neumann algebras  $$l_K: \mathcal{P}_{2,+}(N \subseteq K) \to e^M_K \mathcal{P}_{2,+}(N \subseteq M) e^M_K,$$ 
$$r_K : \mathcal{P}_{2,+} (K \subseteq M) \to e^M_K * \mathcal{P}_{2,+}(N \subseteq M) * e^M_K,$$ for usual $+$, $\times$ and $()^{\star}$, such that
 $$l_K(e^K_P) = e^M_P \text{ and } r_K(e^M_Q) = e^M_Q.$$
 Moreover, the coproduct $*$ is also preserved by these maps, but up to a multiplicative constant, $|M:K|^{1/2}$ for $l_K$ and $|K:N|^{-1/2}$ for $r_K$.
 Then, $\forall m \in \{l^{\pm 1}_K, r^{\pm 1}_K \}$, $\forall a_i > 0$ in the domain of $m$, $m(a_i)>0$ and $$\langle m(a_1), \dots, m(a_n) \rangle = m (\langle a_1, \dots, a_n \rangle).$$
\end{theorem}
\begin{proof} 
Immediate from \cite{kcb} or \cite{la0}, using Lemma \ref{prodcoprod}. We can compute the multiplicative constant for $l_K$ on the coproduct, directly as follows. Let $\alpha$ be the constant such that for any $a,b \in \mathcal{P}_{2,+}(N \subseteq K)$, $$l_K(a*b) = \alpha l_K(a) * l_K(b).$$ Note that $l_K(e^K_K) = e^M_K$, $e^M_K * e^M_K = |M:N|^{1/2} |M:K|^{-1}e^M_K$ and  $$l_K(e^K_K * e^K_K) = |K:N|^{1/2} e^M_K.$$ So, $$\alpha = |M:K| |K:N|^{1/2} / |M:N|^{1/2} = |M:K|^{1/2}.$$ We can compute similarly the constant for $r_K$.
\end{proof}    
\begin{notations} Let  $b_1 \le b \le b_2$ be the biprojections $e^M_P \le e^M_K \le e^M_Q$. We define $ l_b:=l_K $ and $ r_b:=r_K $; also $ \mathcal{P}(b_1, b_2):= \mathcal{P}(P \subseteq Q) $ and $$ |b_2:b_1|:=\tr(b_2)/\tr(b_1)=|Q:P|. $$
\end{notations}

\section{Ore's theorem on subfactor planar algebras}

We will generalize Theorem \ref{ore} to any irreducible subfactor planar algebra $\mathcal{P}$. The proof (organized in lemmas and propositions) is inspired by the proof of \cite[Theorem 4.9]{p2} and Clifford theory.

\begin{proposition} \label{mini}
Let $p \in \mathcal{P}_{2,+}$ be a minimal central projection. Then, there exists $u \le p$ minimal projection such that $\langle u \rangle = \langle p \rangle$.
\end{proposition}
\begin{proof}
 If $p$ is a minimal projection, then it's ok. Else, let $b_1, \dots , b_n$ be the coatoms of $[e_1,\langle p \rangle]$ ($n$ is finite by \cite{wa}). If $p \not \preceq \sum_{i=1}^n b_i$ then $\exists u \le p $ minimal projection such that $u  \not \le b_i \ \forall i$, so that $\langle u \rangle = \langle p \rangle$. Else $p \preceq \sum_{i=1}^n b_i$ (with $n>1$, otherwise $p \le b_1$ and $\langle p \rangle \le b_1$, contradiction). Consider $E_i=\im(b_i)$ and $F=\im(p)$, then $F=\sum_i E_i \cap F$ (because $p$ is a minimal central projection) with $1<n<\infty$ and $E_i \cap F \subsetneq F \ \forall i$ (otherwise $\exists i$ with $p \le b_i$, contradiction), so $\dim (E_i \cap F)< \dim (F)$ and there exists $U \subseteq F$ one-dimensional subspace such that $U \not \subseteq E_i \cap F$  $\forall i$, and so a fortiori $U \not \subseteq E_i$   $\forall i$. It follows that $u=p_{U} \le p$  is a minimal projection such that $\langle u \rangle = \langle p \rangle$.
\end{proof} 

\noindent Thanks to Proposition \ref{mini}, we can give the following definition:  

\begin{definition} \label{weakly} A planar algebra $\mathcal{P}$ is \emph{weakly cyclic} (or \emph{w-cyclic}) if it satisfies one of the following equivalent assertions:  
\begin{itemize}
\item $\exists u \in  \mathcal{P}_{2,+}$  minimal projection  such that $\langle u \rangle=\id$,
\item $\exists p \in  \mathcal{P}_{2,+}$  minimal central projection  such that $\langle p \rangle=\id$.
\end{itemize}
Moreover, $(N \subseteq M)$ is called w-cyclic if its planar algebra is w-cyclic.
\end{definition}

Let $\mathcal{P} = \mathcal{P}(N \subseteq M)$ be an irreducible subfactor planar algebra. Take an intermediate subfactor $N \subseteq K \subseteq M$ and its biprojection $b = e^M_K$.

\begin{proposition} \label{left}
The planar algebra $\mathcal{P}(e_1,b)$ is w-cyclic if and only if there is a minimal projection $u \in \mathcal{P}_{2,+}$ such that $\langle u \rangle = b$.
\end{proposition}
\begin{proof}
The planar algebra $\mathcal{P}(N \subseteq K)$ is w-cyclic if and only if there is a minimal projection $x \in \mathcal{P}_{2,+}(N \subseteq K)$ such that $\langle x \rangle =  e^K_K$, if and only if $l_K(\langle x \rangle) = l_K(e^K_K)$, if and only if $\langle u \rangle  =  e^M_K$ (by Theorem \ref{2iso}), with $u=l_K(x)$ a minimal projection in $e^M_K\mathcal{P}_{2,+}e^M_K$ and in $\mathcal{P}_{2,+}$.  
\end{proof}

\begin{proposition} \label{right}
The planar algebra $\mathcal{P}(b,\id)$ is w-cyclic if and only if there is a minimal projection $v \in \mathcal{P}_{2,+}$ such that $\langle b,v \rangle = \id$ and $r_b^{-1}(b*v*b)$ is a positive multiple of a minimal projection.
\end{proposition}
\begin{proof}

The planar algebra $\mathcal{P}(K \subseteq M)$ is w-cyclic if and only if there is a minimal projection $x \in \mathcal{P}_{2,+}(K \subseteq M)$ such that $\langle x \rangle =  e^M_M$, if and only if $r_K(\langle x \rangle) = r_K(e^M_M)$, if and only if $\langle r_K(x) \rangle  =  e^M_M$ by Theorem \ref{2iso}. The results follows by Lemmas \ref{lem1} and \ref{lem2}, below.
\end{proof}
\begin{lemma} \label{pos} Let $\mathcal{A}$ be a $\star$-subalgebra of $\mathcal{P}_{2,+}$. Then, any element $x \in \mathcal{A}$ is positive in $\mathcal{A}$ if and only if it is positive in $\mathcal{P}_{2,+}$.  \end{lemma}
\begin{proof}
If $x$ is positive in $\mathcal{A}$, then it is of the form $aa^{\star}$, with $a \in \mathcal{A}$, but $a \in \mathcal{P}_{2,+}$ also, so $x$ is positive in $\mathcal{P}_{2,+}$. Conversely, if $x$ is positive in $\mathcal{P}_{2,+}$ then $\langle xy | y \rangle=\tr(y^{\star}xy) \ge 0$, for any $y \in \mathcal{P}_{2,+}$, so in particular, for any $y \in \mathcal{A}$, which means that $x$ is positive in $\mathcal{A}$.
\end{proof}
\noindent Note that Lemma \ref{pos} will be applied to $\mathcal{A} = b\mathcal{P}_{2,+}b$ or $b*\mathcal{P}_{2,+}*b$.
\begin{lemma} \label{lem1}
For any minimal projection $x \in \mathcal{P}_{2,+}(b,\id)$, $r_b(x)$ is positive and for any minimal projection $v \preceq r_b(x)$, there is $\lambda > 0$ such that $ b * v * b = \lambda r_b(x)$.
\end{lemma}
\begin{proof}
First $x$ is positive, so by Theorem \ref{2iso}, $r_b(x)$ is also positive. For any minimal projection $v \preceq r_b(x)$, we have $b * v * b \preceq  r_b(x)$, because $$b * v * b \preceq b*r_b(x)*b =  b*b * u * b*b \sim b*u*b = r_b(x),$$ by Lemma \ref{th}(2) and with $u \in \mathcal{P}_{2,+}$.  Now by Lemma \ref{th}(1), $b*v*b>0$, so $r_b^{-1}(b*v*b)>0$ also, and by Theorem \ref{2iso}, $$r_b^{-1}(b * v * b) \preceq  x.$$ But $x$ is a minimal projection, so by positivity, $\exists \lambda > 0$ such that $$r_b^{-1}(b * v * b) = \lambda x.$$ It follows that $ b * v * b = \lambda r_b(x)$.
\end{proof}

\begin{lemma} \label{lem2}
For $v \in \mathcal{P}_{2,+}$ positive, $\langle b*v*b \rangle = \langle b,v \rangle$.
\end{lemma}
\begin{proof}
First, by Definition \ref{gener}, $b * v * b  \preceq \langle b,v \rangle$, so by Lemma \ref{th}(3), $\langle b * v * b \rangle \le \langle b,v \rangle$. Next $e_1 \le b$ and $x*e_1 = e_1*x = \delta^{-1}x $, so $$v = \delta^2 e_1 * v * e_1 \preceq b * v * b.$$ Moreover by Theorem \ref{biproj}, $\overline{v} \preceq \langle b * v * b \rangle$, but by Lemma \ref{pre2}, $$\overline{v} * b * v * b \succeq \overline{v} * e_1 * v * b \sim \overline{v} * v * b \succeq e_1 * b \sim b.$$ Then   $b,v  \le  \langle b * v * b \rangle$, so we also have $\langle b,v \rangle \le \langle b * v * b \rangle$.  
\end{proof}

\begin{proposition} \label{topw}
Let $\mathcal{P}$ be an irreducible subfactor planar algebra and $[e_1,\id]$ its biprojection lattice. Let $[t, \id]$ be the top interval of $[e_1, \id]$. Then, $\mathcal{P}$ is w-cyclic if $\mathcal{P}(t,\id)$ is so.
\end{proposition}
\begin{proof}
Let $b_1, \dots , b_n$ be the coatoms of $[e_1, \id]$ and $t = \bigwedge_{i=1}^n b_i$. By assumption and Proposition \ref{right}, there is a minimal projection $v \in \mathcal{P}_{2,+}$ with $\langle t,v \rangle = \id$. If $\exists i$ such that $v \le b_i$, then $\langle t,v \rangle \le b_i$, contradiction. So $\forall i $, $v \not \le b_i$ and then $\langle v \rangle = \id$. 
\end{proof}

\begin{theorem} \label{mainsubfactor}
An irreducible subfactor planar algebra $\mathcal{P}$ with a top Boolean biprojection lattice $[e_1, \id]$ is w-cyclic.
\end{theorem}
\begin{proof}
By Proposition \ref{topw}, we can assume $[e_1, \id]$ Boolean.

\noindent We will make a proof by induction on the rank of the Boolean lattice.  

\noindent If $[e_1, \id]$ is of rank $1$, then for any minimal projection $u \neq e_1$, $\langle u \rangle = \id$. Now suppose $[e_1, \id]$ Boolean of rank $n>1$, and assume the result true for any rank $<n$. Let $b$ be an atom of $[e_1, \id]$. Then $[b, \id]$ is Boolean of rank $n-1$, so by assumption $\mathcal{P}(b,\id)$ is w-cyclic, thus there is a minimal projection $x \in \mathcal{P}_{2,+}(b,\id)$ with $\langle x \rangle = \id$. By Theorem \ref{2iso},  Lemmas \ref{th}, \ref{lem1} and \ref{lem2}, for any minimal projection $v \preceq r_b(x)$, $$ b \vee \langle v \rangle =  \langle b, v \rangle =  \langle b*v*b \rangle = \langle r_b(x) \rangle =r_b(\langle x \rangle) = r_b(\id) =  \id.$$ 
Thus $\langle v \rangle \in \{b^{\complement} , \id\}$ by Lemma \ref{compl}. 
\noindent Assume that for every minimal projection $v \preceq r_b(x)$ we have $\langle v \rangle =  b^{\complement}$; because $ r_b(x) > 0$, by the spectral theorem, there is an integer $m$ and minimal projections $v_1, \dots , v_m$ such that $r_b(x) \sim \sum_{i=1}^m v_i$, so 
$$\id  = \langle r_b(x) \rangle = \langle \sum_{i=1}^m v_i \rangle \le  \bigvee_{i=1}^m \langle v_i \rangle =  \bigvee_{i=1}^m b^{\complement} = b^{\complement},$$ 
thus $\id \le b^{\complement}$, contradiction. So there is a minimal projection $v \preceq r_b(x)$ such that $\langle v \rangle = \id $, and the result follows. 
 \end{proof}
 
\noindent The proof of Theorem \ref{mainintro} follows by Proposition \ref{topBn}. 
In general, we deduce the following non-trivial upper bound:
\begin{corollary} \label{upper}
The minimal number $r$ of minimal projections generating the identity biprojection of $\mathcal{P}$ (i.e. $\langle u_1, \dots, u_r \rangle = \id$) is at most the minimal length $\ell$ for an ordered chain of biprojections $$e_1=b_0 < b_1 < \dots < b_{\ell} = \id$$ such that $[b_i,b_{i+1}]$ is top Boolean.
\end{corollary}
\begin{proof} Immediate from Theorem \ref{mainsubfactor} and Lemma \ref{interw}.
\end{proof}
\begin{remark} Let $(N \subset M)$ be an irreducible finite index subfactor. Then Corollary \ref{upper} reformulates as a non-trivial upper bound for the minimal number of (algebraic) irreducible sub-$N$-$N$-bimodules of $M$, generating $M$ as von Neumann algebra.
\end{remark} 
\begin{lemma} \label{interw}
Let $b' < b$ be biprojections. If $\mathcal{P}(b' , b)$ is w-cyclic, then there is a minimal projection $u \in \mathcal{P}_{2,+}$ such that $\langle b', u \rangle = b$.  
\end{lemma}
\begin{proof} Take the von Neumann algebras isomorphisms (Theorem \ref{2iso}) $$l_{b}: \mathcal{P}_{2,+}(e_1 , b) \to b\mathcal{P}_{2,+}b$$ and, with $a = l_{b}^{-1}(b')$, $$r_{a}: \mathcal{P}_{2,+}(b' , b) \to a * \mathcal{P}_{2,+}(e_1 , b) * a.$$ Then, by assumption, the planar algebra $\mathcal{P}(b' , b)$ is w-cyclic, so by Proposition \ref{right}, $\exists u' \in \mathcal{P}_{2,+}(e_1,b)$ minimal projection such that $$\langle a, u' \rangle = l_b^{-1}(b).$$ Then by applying the map $l_{b}$ and Theorem \ref{2iso}, we get $$b=\langle l_{b}(a), l_{b}(u') \rangle = \langle b', u \rangle$$ with $u=l_{b}(u')$ a minimal projection of $b\mathcal{P}_{2,+}b$, so of $\mathcal{P}_{2,+}$.
\end{proof}
\noindent We can idem assume $r_a^{-1}(a*l_b^{-1}(u)*a)$ minimal projection of $\mathcal{P}_{2,+}(b' , b)$.

 \section{Applications to finite group theory}

We will give several group theoretic translations of Theorem \ref{mainsubfactor} and Corollary \ref{upper}, giving a new link between combinatorics and representations in finite group theory. Let $G$ be a finite group acting outerly on the hyperfinite ${\rm II}_1$ factor $R$.
Note that the subfactor $(R \subseteq R \rtimes G)$ is w-cyclic if and only if $G$ is cyclic. More generally, for $H$ a subgroup of $G$:
\begin{theorem} \label{rwgrp}
The subfactor $(R \rtimes H \subseteq R \rtimes G)$ is w-cyclic if and only if $[H,G]$ is $H$-cyclic.
\end{theorem}
\begin{proof}
By Proposition \ref{right},  $(R \rtimes H \subseteq R \rtimes G)$ is w-cyclic if and only if $$\exists u \in \mathcal{P}_{2,+}(R \subseteq R \rtimes G) \simeq \bigoplus_{g \in G} \mathbb{C}e_g \simeq \mathbb{C}^G$$ minimal projection such that  $\langle b , u \rangle = \id$ with $b=e^{R \rtimes G}_{R \rtimes H}$ and $r_b^{-1}(b*u*b)$ minimal projection, if and only if $\exists g \in G$ such that $\langle H,g \rangle = G$, because $u$ is of the form $e_g$ and  $\forall g' \in HgH$, $Hg'H = HgH$.  \end{proof}

\noindent Theorem \ref{ore} is a translation of Theorem \ref{mainsubfactor} using Theorem \ref{rwgrp}.

\begin{corollary} \label{gene}
The minimal cardinal for a generating set of $G$ is at most the minimal length $\ell$ for an ordered chain of subgroups  $$\{e\}=H_0 < H_1 < \dots < H_{\ell} = G$$  such that $[H_i,H_{i+1}]$ is top Boolean.
\end{corollary}
\begin{proof} Immediate from Corollary \ref{upper} and Theorem \ref{rwgrp}.  \end{proof}

 \begin{definition} \label{fixstab} Let $W$ be a representation of a group $G$, $K$ a subgroup of $G$, and $X$ a subspace of $W$. Let the \emph{fixed-point subspace} be $$W^{K}:=\{w \in W \ \vert \  kw=w \ , \forall k \in K  \}$$ and the \emph{pointwise stabilizer subgroup} $$G_{(X)}:=\{ g \in G \  \vert \ gx=x \ , \forall x \in X \}$$  \end{definition} 

\begin{lemma} \label{corrminstab}
Let $p \in \mathcal{P}_{2,+}(R^G \subseteq R)$ be the projection on a space $X$ and $b$ the biprojection of a subgroup $H$ of $G$. Then $$p \le b  
\Leftrightarrow H \subseteq G_{(X)}.$$ It follows that the biprojection $\langle p \rangle$ corresponds to the subgroup $G_{(X)}$.
\end{lemma}
\begin{proof} First, $p \le b $ if and only if $ \forall x \in X, bx=x $. Now, there is $\lambda > 0$ such that $$\mathcal{F}^{-1}(b) = \lambda \sum_{h \in H} e_h.$$ So, if $bx=x$, then $\mathcal{F}(e_h) x = \mathcal{F}(e_h) (bx) =  (\mathcal{F}(e_h)b)x = bx = x.$   Thus,  $$p \le b   \Leftrightarrow \forall h \in H, \ \forall x \in X, \   \mathcal{F}(e_h) x = x  \Leftrightarrow  \forall h \in H, \ h \in G_{(X)}.$$ The result follows.   
\end{proof}

\noindent We deduce an amusing alternative proof of a well-known result \cite[\S 226]{bu}:

\begin{corollary} A complex representation $V$ of a finite group $G$ is faithful if and only if for any irreducible complex representation $W$ there is an integer $n$ such that $W \preceq V^{\otimes n}$.
\end{corollary}  
\begin{proof} Let $V$ be a complex representation of $G$, and let $V_1,\dots, V_s$ be (equivalence class representatives of) its irreducible components. Then $$\ker(\pi_V) = \bigcap_{i=1}^s \ker(\pi_{V_i})= \bigwedge_{i=1}^s G_{(V_i)}.$$ Now, $V$ is faithful if and only if $\ker(\pi_V) = \{ e\}$, if and only if (by Lemma \ref{corrminstab}) $\bigvee_{i=1}^s \langle p_i \rangle = \id$, with $p_i \in \mathcal{P}_{2,+}(R^G \subseteq R) \simeq \mathbb{C}G$, the minimal central projection on $V_i$. But by Definition \ref{gener}, $\langle p_1, \dots, p_s \rangle$ is the range projection of $$\sum_{n=1}^N(p_1 + \cdots + p_s)^{*n}$$ for $N$ large enough.  The result follows by Corollary \ref{coprofusion}.
\end{proof}

\begin{definition} \label{linprimdef}
The group $G$ is called \emph{linearly primitive} if it admits an irreducible complex representation $V$ which is faithful, i.e. $G_{(V)} = \{ e \}$.
\end{definition}

\begin{definition} \label{linprim}
The interval $[H,G]$ is called \emph{linearly primitive} if there is an irreducible complex representation $V$ of $G$ such that $G_{(V^H)} = H$.
\end{definition}

\begin{theorem} \label{lwgrp}
The subfactor $(R^G \subseteq R^H)$ is w-cyclic if and only if $[H,G]$ is linearly primitive.
\end{theorem}
\begin{proof}
By Proposition \ref{left}, $(R^G \subseteq R^H)$ is w-cyclic if and only if $$\exists u \in \mathcal{P}_{2,+}(R^G \subseteq R) \simeq \mathbb{C}G$$ minimal  projection such that  $\langle u \rangle = e^R_{R^H}$, if and only if,  by Lemma \ref{corrminstab}, $H=G_{(U)}$ with $U = \im(u)$. Let $p$ be the central support of $u$, and $V$ its range (irreducible). Then $H \subset G_{(V^H)} \subset G_{(U)}$, so $H = G_{(V^H)}$.   \end{proof}

\begin{corollary} \label{wgrp2}  The subfactor $(R^G \subseteq R)$ is w-cyclic if and only if $G$ is  linearly primitive. \end{corollary}  


We will prove the dual versions of Theorem \ref{ore} and Corollary \ref{gene}, giving the link between combinatorics and representations theory. 

\begin{corollary} \label{dualore}  Let $[H,G]$ be a bottom Boolean interval of finite groups. Then $\exists V$ irreducible complex representation of $G$ such that $G_{(V^H)}=H$.
\end{corollary}
\begin{proof} It is the group theoretic reformulation of Theorem \ref{mainsubfactor} for $\mathcal{P}(R^G  \subseteq R^H)$, thanks to Theorem \ref{lwgrp}. 
\end{proof}
\noindent The proof of Theorem \ref{dualoreintro} follows by Proposition \ref{topBn}.
\begin{corollary} \label{dualgene}
The minimal number of irreducible components for a faithful complex representation of $G$ is at most the minimal length $\ell$ for an ordered chain of subgroups $$\{e\}=H_0 < H_1 < \dots < H_{\ell} = G$$ such that $[H_i,H_{i+1}]$ is bottom Boolean.
\end{corollary}
\begin{proof} It's a reformulation of Corollary \ref{upper} for $\mathcal{P} = \mathcal{P}(R^G  \subseteq R)$, using Definition \ref{gener}, Proposition \ref{mini} and the fact that the coproduct of two minimal central projections of $\mathcal{P}_{2,+}$ is given by the tensor product of the associated irreducible representations of $G$, by Corollary \ref{coprofusion}. \end{proof}
\begin{definition} \label{core} A subgroup $H$ of a group $G$ is called \emph{core-free} if any normal subgroup of $G$ contained in $H$ is trivial. 
\end{definition}
\noindent Note that Corollary \ref{dualgene} can be improved be taking for $H_0$ any core-free subgroup of $H_1$, thanks to the following lemma.
\begin{lemma} \label{linprimgrp} For $H$ core-free, $G$ is linearly primitive if $[H,G]$ is so. 
\end{lemma}
\begin{proof} Take $V$ as above. Now, $V^{H} \subset V$ so $G_{(V)} \subset G_{(V^H)}$, but $\ker(\pi_V) =  G_{(V)}$, it follows that $\ker(\pi_V) \subset H$; but $H$ is a core-free subgroup of $G$, and $\ker(\pi_V)$ a normal subgroup of $G$, so $\ker(\pi_V)= \{ e \}$, which means that $V$ is faithful on $G$, i.e. $G$ is linearly primitive. \end{proof}

\begin{remark} We get as well a non-trivial upper bound for the minimal number of irreducible components for a faithful (co)representation of a finite dimensional Kac algebra, involving the lattice of left coideal $\star$-subalgebras, by Galois correspondence \cite[Theorem 4.4]{ilp}. 
\end{remark}

\section{Appendix}
In the irreducible depth $2$ case, the relation between coproduct and fusion rules is well-known to experts, but we did not find a proof in the literature. Because we need it in the proof of Corollary \ref{dualgene}, for the convenience of the reader, we will prove this relation in this appendix.

Let $\mathcal{P}$ be a subfactor planar algebra which is irreducible and depth $2$, i.e. $\mathcal{P}_{1,+} = \mathbb{C}$ and $\mathcal{P}_{3,+}$ is a factor. By \cite[Section 3]{dk}, $\mathcal{P}=\mathcal{P} (R^{\mathbb{A}} \subseteq R)$, with $\mathbb{A}$ a Kac algebra equal to $\mathcal{P}_{2,+}$ and acting outerly on the hyperfinite ${\rm II}_1$ factor $R$.  

\begin{theorem}[Splitting, \cite{kls} p39] \label{split} Any element $x \in \mathbb{A}$ splits as follows:
$$
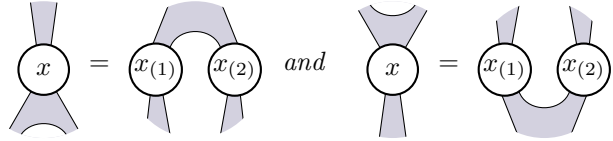
\begin{tikzpicture}[scale=.5, PAdefn]
	\clip (0,0) circle (2.6cm);
	\draw[shaded] (-0.15,0) -- (100:4cm) -- (80:4cm) -- (0.15,0);
	\draw[shaded] (0,0) -- (-120:4cm) -- (-60:4cm) -- (0,0);
	\node at (0,0) [Tbox, inner sep=1.5mm] {{$x$} };
 \draw[fill=white] (-0.7,-4) .. controls ++(120:3cm) and ++(60:3cm) .. (0.7,-4);	
\end{tikzpicture}  
\hspace{-0.4cm} =  \  	
	\begin{tikzpicture}[scale=.5, PAdefn]
	\clip (0,0) circle (2.6cm);
	\draw[shaded] (-2,0) -- (0,5) -- (2,0) -- (-2,0);
	\draw[fill=white] (-0.7,-0.5) .. controls ++(120:3cm) and ++(60:3cm) .. (0.7,-0.5);
	\draw[shaded] (-1.6,0) -- (-120:4cm) -- (-100:4cm) -- (-1.4,0);
	\draw[shaded] (1.4,0) -- (-80:4cm) -- (-60:4cm) -- (1.6,0);
	\node at (-1.5,0) [Tbox, inner sep=0.2mm] {{$x_{(1)}$} };
	\node at (1.5,0) [Tbox, inner sep=0.2mm] {{$x_{(2)}$} };
\end{tikzpicture} 
\ \text{ and}  
\begin{tikzpicture}[scale=.5, PAdefn]
	\clip (0,0) circle (2.6cm);
	\draw[shaded] (-0.15,0) -- (260:4cm) -- (280:4cm) -- (0.15,0);
	\draw[shaded] (0,0) -- (60:4cm) -- (120:4cm) -- (0,0);
	\node at (0,0) [Tbox, inner sep=1.5mm] {{$x$} };
 \draw[fill=white] (-0.7,4) .. controls ++(240:3cm) and ++(300:3cm) .. (0.7,4);	
\end{tikzpicture} 
\hspace{-0.4cm} =  \ 	
\begin{tikzpicture}[scale=.5, PAdefn]
	\clip (0,0) circle (2.6cm);
	\draw[shaded] (-2,0) -- (0,-5) -- (2,0) -- (-2,0);
	\draw[fill=white] (-0.7,0.5) .. controls ++(240:3cm) and ++(300:3cm) .. (0.7,0.5);
	\draw[shaded] (-1.6,0) -- (-240:4cm) -- (-260:4cm) -- (-1.4,0);
	\draw[shaded] (1.4,0) -- (80:4cm) -- (60:4cm) -- (1.6,0);
	\node at (-1.5,0) [Tbox, inner sep=0.2mm] {{$x_{(1)}$} };
	\node at (1.5,0) [Tbox, inner sep=0.2mm] {{$x_{(2)}$} };
\end{tikzpicture}$$ 
\noindent with $\Delta(x) = x_{(1)} \otimes x_{(2)}$ the sumless Sweedler notation of the comultiplication. 
\end{theorem}

\begin{corollary} \label{thmZZ}
If $a,b \in  \mathbb{A}$ are central, then so is $a * b$.
\end{corollary}
\begin{proof}
 This diagrammatic proof by splitting is due to Vijay Kodiyalam.
$$ \hspace*{0.2cm} (a*b) \cdot x =  \hspace{-0.3cm}
\begin{tikzpicture}[scale=.5, PAdefn]
	\clip (0,0) circle (3cm);
	\draw[shaded] (-0.15,0) -- (100:4cm) -- (80:4cm) -- (0.15,0);
	\draw[shaded] (0,1.2) -- (-130:4cm) -- (-50:4cm) -- (0,1.2);
	\draw[shaded] (0,1.2) -- (-130:4cm) -- (-50:4cm) -- (0,1.2);
	\node at (0,1.2) [Tbox, inner sep=1.35mm] {{$x$} };
	\draw[fill=white] (0,-1.2) circle (0.75cm);
	\node at (-1.2,-1.2) [Tbox, inner sep=1.5mm] {{$a$} };
	\node at (1.2,-1.2) [Tbox, inner sep=1.35mm] {{$b$} };
\end{tikzpicture}  
\hspace{-0.4cm} =     \textbf{\hspace{-0.2cm}}	
	\begin{tikzpicture}[scale=.5, PAdefn]
	\clip (0,0) circle (3cm);
	\draw[shaded] (-2,1) -- (0,5) -- (2,1) -- (-2,1);
	\draw[shaded] (-2,-1) -- (0,-5) -- (2,-1) -- (-2,-1);
	\draw[shaded] (-1.65,-1.2) -- (1.65,-1.2) -- (1.65,1.4) -- (-1.65,1.4) -- (-1.65,-1.2);
	\draw[fill=white] (0,0) ellipse (0.8cm and 2.3cm);
	\node at (-1.2,-1.2) [Tbox, inner sep=1.5mm] {{$a$} };
	\node at (1.2,-1.2) [Tbox, inner sep=1.35mm] {{$b$} };
	\node at (-1.2,1.4) [Tbox, inner sep=0.1mm] {{$x_{(1)}$} };
	\node at (1.2,1.4) [Tbox, inner sep=0.1mm] {{$x_{(2)}$} };
\end{tikzpicture} 
\hspace{-0.1cm} =    \hspace{-0.2cm}	
	\begin{tikzpicture}[scale=.5, PAdefn]
	\clip (0,0) circle (3cm);
	\draw[shaded] (-2,1) -- (0,5) -- (2,1) -- (-2,1);
	\draw[shaded] (-2,-1) -- (0,-5) -- (2,-1) -- (-2,-1);
	\draw[shaded] (-1.65,-1.2) -- (1.65,-1.2) -- (1.65,1.4) -- (-1.65,1.4) -- (-1.65,-1.2);
	\draw[fill=white] (0,0) ellipse (0.8cm and 2.3cm);
	\node at (-1.2,-1.2) [Tbox, inner sep=0.1mm] {{$x_{(1)}$} };
	\node at (1.2,-1.2) [Tbox, inner sep=0.1mm] {{$x_{(2)}$} };
	\node at (-1.2,1.4) [Tbox, inner sep=1.5mm] {{$a$} };
	\node at (1.2,1.4) [Tbox, inner sep=1.35mm] {{$b$} };
\end{tikzpicture} 
  \hspace{-0.05cm} =  \hspace{-0.3cm}
\begin{tikzpicture}[scale=.5, PAdefn]
	\clip (0,0.2) circle (3cm);
	\draw[shaded] (-0.15,2) -- (-80:4.2cm) -- (-100:4.2cm) -- (0.15,2);
	\draw[shaded] (0,-1) -- (50:5cm) -- (130:5cm) -- (0,-1);
	\node at (0,-1) [Tbox, inner sep=1.35mm] {{$x$} };
	\draw[fill=white] (0,1.4) circle (0.75cm);
	\node at (-1.2,1.4) [Tbox, inner sep=1.5mm] {{$a$} };
	\node at (1.2,1.4) [Tbox, inner sep=1.35mm] {{$b$} };
\end{tikzpicture}  
\hspace{-0.3cm} = x \cdot (a*b) \hspace*{0.2cm} \qedhere $$
 \end{proof}  

 \begin{proposition} \label{tr(x(a*b))}
 Consider $a,b,x \in \mathbb{A}$. Then $\langle a * b \vert x\rangle  =  \delta \langle a \otimes b \vert \Delta(x) \rangle$.
 \end{proposition}  
\begin{proof} By irreducibility, $\tr(x*y) = \delta \tr(x)\tr(y)$. Then, by Theorem \ref{split},
$$\tr(x^{\star} (a*b)) = \tr((x^{\star}_{(1)}a)*(x^{\star}_{(2)}b)) = \delta \tr(x^{\star}_{(1)}a) \tr(x^{\star}_{(2)}b).$$
The result follows by definition, $\langle x \vert y \rangle:= \tr(y^{\star}x)$ and  $\langle a \otimes b \vert c \otimes d  \rangle := \langle a \vert c \rangle \langle b \vert d \rangle$.
\end{proof} 

\noindent Note that as a finite dimensional von Neumann algebra, $$\mathbb{A} \simeq \bigoplus_i End(H_i).$$
As a Kac algebra, it acts on a tensor product $V \otimes W$ as follows:   
   $$\Delta(x) (v\otimes w) = (x_{(1)} v)\otimes (x_{(2)} w),$$
and $H_{i}\otimes H_{j}$ decomposes into irreducible representations 
$$ H_{i}\otimes H_{j} = \bigoplus_{k}{M_{ij}^{k} \otimes H_{k}}    $$   
with $M_{ij}^{k}$ the multiplicity space. It follows that
$$n_{i}  n_{j} = \sum_k n_{ij}^{k}  n_{k}$$
with $n_k = \dim (H_{k})$ and $n_{ij}^{k} = \dim (M_{ij}^{k})$.  The following proposition gives the relation between comultiplication and fusion rules $(n_{ij}^{k})$.

\begin{proposition} \label{thmcopro} The inclusion matrix of the unital inclusion of finite dimensional von Neumann algebras $\Delta(\mathbb{A}) \subseteq \mathbb{A} \otimes \mathbb{A}$ is $\Lambda = (n_{ij}^{k})$.  \end{proposition}
\begin{proof}
The irreducible representations of $\mathbb{A} \otimes \mathbb{A}$ are $(H_{i}\otimes H_{j})_{i,j}$, so by the double commutant theorem and the Schur's lemma, we get that   
$$\pi_{H_{i}\otimes H_{j}}(\mathbb{A} \otimes \mathbb{A}) = \pi_{H_{i}\otimes H_{j}}(\mathbb{A} \otimes \mathbb{A})'' = End(H_{i}\otimes H_{j}) \simeq M_{n_{i}n_{j}}(\mathbb{C}).$$ 
Moreover, by the fusion rules
$$\pi_{H_{i}\otimes H_{j}}(\Delta(\mathbb{A})) \simeq \bigoplus_{k}{M_{ij}^{k} \otimes \pi_{H_{k}}(\mathbb{A})}  \simeq \bigoplus_{k}{M_{ij}^{k} \otimes M_{n_{k}}(\mathbb{C})}.$$  
So, the inclusion matrix of the following inclusion is $(n_{ij}^k)_k$. 
$$\pi_{H_{i}\otimes H_{j}}(\Delta(\mathbb{A})) \subseteq \pi_{H_{i}\otimes H_{j}}(\mathbb{A} \otimes \mathbb{A}).$$
Take $V = \bigoplus_{i,j} H_{i}\otimes H_{j}$. Then, we have the isomorphism of inclusions:  
$$[\Delta(\mathbb{A})  \subseteq \mathbb{A}\otimes \mathbb{A}] \simeq [\pi_{V}(\Delta(\mathbb{A})) \subseteq \pi_{V}(\mathbb{A}\otimes \mathbb{A})].$$  
But $\pi_V =  \bigoplus_{i,j} \pi_{H_{i}\otimes H_{j}}$, so the result follows.   \end{proof} 

\begin{corollary} \label{coprofusion} Let $p_i \in \mathbb{A}$ be the minimal central projection on $H_i$. The relation between coproduct and fusion rules is the following: $$p_i * p_j = \delta \sum_k n_{ij}^{k} p_k.$$ \end{corollary} 
\begin{proof} By Corollary \ref{thmZZ}, there is $\epsilon_{ij}^{k} \ge 0$ such that $p_i * p_j = \sum_k \epsilon_{ij}^{k} p_k.$
So, $ \langle p_i * p_j  \vert p_k \rangle = \epsilon_{ij}^{k} \tr(p_k)$. But, by Propositions \ref{tr(x(a*b))} and \ref{thmcopro}, $$\langle p_i * p_j  \vert p_k \rangle = \delta \langle p_i \otimes p_j | \Delta(p_k) \rangle = \delta n_{ij}^{k} \tr(p_k).$$
The result follows. \end{proof}

\section{Acknowledgments} 
The author would like to thank the  Isaac  Newton  Institute  for  Mathematical  Sciences,  Cambridge,  for  support and  hospitality  during  the  programme  \textit{Operator  Algebras:  Subfactors  and  their  applications}, where work on this paper was undertaken. This work was supported by EPSRC grant no EP/K032208/1. Thanks to Vaughan Jones, V.S. Sunder, David Penneys, Vijay Kodiyalam, Zhengwei Liu and Keshab Bakshi for their interest on this work, fruitful exchanges and advice.


\begin{bibdiv}
\begin{biblist}
\bib{kcb}{article}{
   author={Bakshi, Keshab Chandra},
   title={Intermediate planar algebra revisited},
   journal={Internat. J. Math.},
   volume={29},
   number={12},
   year={2018},
   pages={42pp},
   doi={10.1142/S0129167X18500775},
} 
\bib{bdlr}{article}{
   author={Bakshi, Keshab Chandra},
   author={Das, Sayan},
   author={Liu, Zhengwei},
   author={Ren, Yunxiang},
   title={An angle between intermediate subfactors and its rigidity},
   journal={Trans. Amer. Math. Soc.},
   volume={},
   date={},
   number={},
   pages={},
   doi={10.1090/tran/7738},
}
\bib{bp}{article}{
   author={Balodi, Mamta},
   author={Palcoux, Sebastien},
   title={On Boolean intervals of finite groups},
   journal={J. Comb. Theory, Ser. A}
   volume={157}
   pages={49-69},
   date={2018},
   doi={10.1016/j.jcta.2018.02.004}, 
} 
\bib{bi}{article}{
   author={Bisch, Dietmar},
   title={A note on intermediate subfactors},
   journal={Pacific J. Math.},
   volume={163},
   date={1994},
   number={2},
   pages={201--216},
   doi={10.2140/pjm.1994.163.201}
}
\bib{br}{article}{
   author={Brown, Kenneth S.},
   title={The coset poset and probabilistic zeta function of a finite group.},
   journal={J. Algebra},
   volume={225},
   date={2000},
   number={2},
   pages={989--1012},
   doi={10.1006/jabr.1999.8221},
}
\bib{bu}{book}{
   author={Burnside, William},
   title = {Theory of groups of finite order. Second edition},
   publisher={Cambridge University Press},
   date={1911},
   pages={xxiv+512},
}
\bib{dk}{article}{
   author={Das, Paramita},
   author={Kodiyalam, Vijay},
   title={Planar algebras and the Ocneanu-Szyma\'nski theorem},
   journal={Proc. Amer. Math. Soc.},
   volume={133},
   date={2005},
   number={9},
   pages={2751--2759},
   doi={10.1090/S0002-9939-05-07789-0},
}
\bib{ilp}{article}{
   author={Izumi, Masaki},
   author={Longo, Roberto},
   author={Popa, Sorin},
   title={A Galois correspondence for compact groups of automorphisms of von
   Neumann algebras with a generalization to Kac algebras},
   journal={J. Funct. Anal.},
   volume={155},
   date={1998},
   number={1},
   pages={25--63},
   doi={10.1006/jfan.1997.3228},
}
\bib{jlw}{article}{
   author={Jiang, Chunlan},
   author={Liu, Zhengwei},
   author={Wu, Jinsong},
   title={Noncommutative uncertainty principles},
   journal={J. Funct. Anal.},
   volume={270},
   date={2016},
   number={1},
   pages={264--311},
   doi={10.1016/j.jfa.2015.08.007}
}
\bib{kls}{article}{
   author={Kodiyalam, Vijay},
   author={Landau, Zeph},
   author={Sunder, V. S.},
   title={The planar algebra associated to a Kac algebra},
   journal={Proc. Indian Acad. Sci. Math. Sci.},
   volume={113},
   date={2003},
   number={1},
   pages={15--51},
   doi={10.1007/BF02829677},
}
\bib{jo}{article}{
 author={Jones, Vaughan F. R.},
 title={Actions of finite groups on the hyperfinite type $ {\rm II}_{1} $ \
 factor},
 journal={Mem. Amer. Math. Soc.},
 volume={28},
 date={1980},
 number={237},
 pages={v+70},
 doi={10.1090/memo/0237},
}
\bib{jo2}{article}{
 author={Jones, Vaughan F. R.},
 title={Index for subfactors},
 journal={Invent. Math.},
 volume={72},
 date={1983},
 number={1},
 pages={1--25},
 doi={10.1007/BF01389127},
}
\bib{js}{book}{
 author={Jones, Vaughan F. R.},
 author={Sunder, V. S.},
 title={Introduction to subfactors},
 series={London Mathematical Society Lecture Note Series},
 volume={234},
 publisher={Cambridge University Press, Cambridge},
 date={1997},
 pages={xii+162},
 doi={10.1017/CBO9780511566219},
}
\bib{jo4}{article}{
 author={Jones, Vaughan F. R.},
 title={Planar algebras, I},
 date={1999},
 pages={122pp},
 note={To appear in New Zealand Journal of Mathematics. arXiv:math/9909027},
}
\bib{sk2}{article}{
 author={Kodiyalam, Vijay},
 author={Sunder, V. S.},
 title={On Jones' planar algebras},
 journal={J. Knot Theory Ramifications},
 volume={13},
 date={2004},
 number={2},
 pages={219--247},
 doi={10.1142/S021821650400310X},
}
\bib{la0}{book}{
   author={Landau, Zeph A.},
   title={Intermediate subfactors},
   note={Thesis (Ph.D.)--University of California at Berkeley},
   date={1998},
   pages={132},
}
\bib{la}{article}{
   author={Landau, Zeph A.},
   title={Exchange relation planar algebras},
   booktitle={Proceedings of the Conference on Geometric and Combinatorial
   Group Theory, Part II (Haifa, 2000)},
   journal={Geom. Dedicata},
   volume={95},
   date={2002},
   pages={183--214},
   doi={10.1023/A:1021296230310},
}
\bib{li}{article}{
   author={Liu, Zhengwei},
   title={Exchange relation planar algebras of small rank},
   journal={Trans. Amer. Math. Soc.},
   volume={368},
   date={2016},
   number={12},
   pages={8303--8348},
   doi={10.1090/tran/6582},
}
\bib{nk}{article}{
 author={Nakamura, Masahiro},
 author={Takeda, Zir{\^o}},
 title={On the fundamental theorem of the Galois theory for finite
 factors. },
 journal={Proc. Japan Acad.},
 volume={36},
 date={1960},
 pages={313--318},
}
\bib{or}{article}{
   author={Ore, Oystein},
   title={Structures and group theory. II},
   journal={Duke Math. J.},
   volume={4},
   date={1938},
   number={2},
   pages={247--269},
   doi={10.1215/S0012-7094-38-00419-3},
}
\bib{p1}{article}{
   author={Palcoux, Sebastien},
   title={Ore's theorem for cyclic subfactor planar algebras and beyond},
   journal={Pacific J. Math.},
   volume={292},
   number={1},
   date={2018},
   pages={203--221},
   doi={10.2140/pjm.2018.292.203}, 
}
\bib{p4}{article}{
   author={Palcoux, Sebastien},
   title={Dual Ore's theorem on distributive intervals of finite groups},
   journal={J. Algebra}
   volume ={505},
   pages = {279--287},
   date = {2018},
   doi={10.1016/j.jalgebra.2018.03.017},
}
\bib{p3}{article}{
   author={Palcoux, Sebastien},
   title={Euler totient of subfactor planar algebras},
   journal={Proc. Am. Math. Soc.}
   volume={146},
   number={11},
   date={2018},
   pages={4775--4786},
   doi={10.1090/proc/14167}
}
\bib{p0}{article}{
   author={Palcoux, Sebastien},
   title={Ore's theorem on cyclic subfactor planar algebras and applications},
   note={arXiv:1505.06649, 50pp, long version of \cite{p1}.},
}
\bib{p2}{article}{
   author={Palcoux, Sebastien},
   title={Dual Ore's theorem for distributive intervals of small index},
   pages={15pp},
   note={arXiv:1610.07253}
}
\bib{sw}{article}{
author = {Shareshian, John},
author = {Woodroofe, Russ},
title = {Order complexes of coset posets of finite groups are not contractible},
journal = {Adv. Math.},
volume ={291},
pages = {758--773},
date = {2016},
doi = {10.1016/j.aim.2015.10.018},
}
\bib{sta}{book}{
   author={Stanley, Richard P.},
   title={Enumerative combinatorics. Volume 1},
   series={Cambridge Studies in Advanced Mathematics},
   volume={49},
   edition={2},
   publisher={Cambridge University Press},
   date={2012},
   pages={xiv+626},
}
\bib{wa}{article}{
   author={Watatani, Yasuo},
   title={Lattices of intermediate subfactors},
   journal={J. Funct. Anal.},
   volume={140},
   date={1996},
   number={2},
   pages={312--334},
   doi={10.1006/jfan.1996.0110},
}
\bib{xu1}{article}{
   author={Xu, Feng},
   title={On representing some lattices as lattices of intermediate subfactors of finite index},
   journal={Adv. Math.},
   volume={220},
   date={2009},
   number={5},
   pages={1317--1356},
   doi={10.1016/j.aim.2008.11.006},
}
\bib{xu2}{article}{
   author={Xu, Feng},
   title={Symmetries of subfactors motivated by Aschbacher-Guralnick conjecture},
   journal={Adv. Math.},
   volume={289},
   date={2016},
   number={},
   pages={345--361},
   doi={10.1016/j.aim.2015.10.029},
}
\end{biblist}
\end{bibdiv}
\end{document}